\documentclass[reqno,12pt]{amsart}
\usepackage{geometry,amsmath,amssymb,bbm,color}
\newcommand{\mathd}{\mathrm{d}}

\definecolor{grey}{rgb}{0.75,0.75,0.75}
\definecolor{orange}{rgb}{1.0,0.5,0.5}
\definecolor{brown}{rgb}{0.5,0.25,0.0}
\definecolor{pink}{rgb}{1.0,0.5,0.5}
\newtheorem{lemma}{Lemma}
\newtheorem{remark}{Remark}
\newtheorem{theorem}{Theorem}

\begin{document}

\title[Global Regularity of 2D Generalized MHD]{A Remark On Global Regularity of
2D Generalized Magnetohydrodynamic Equations}
\author[ Quansen Jiu, Jiefeng Zhao]{}
\maketitle


 \centerline{\scshape Quansen Jiu\footnote{The research is partially
supported by National Natural Sciences Foundation of China (No.
11171229, 11231006 and 11228102) and Project of Beijing Education
Committee.}}

\medskip
{\footnotesize
  \centerline{School of Mathematical Sciences, Capital Normal University}
  \centerline{Beijing  100048, P. R. China}
   \centerline{  \it Email: jiuqs@mail.cnu.edu.cn}

\vspace{3mm}

 \centerline{\scshape Jiefeng Zhao }

\medskip
{\footnotesize
  \centerline{School of Mathematical Sciences, Capital Normal University}
  \centerline{Beijing  100048, P. R. China}
   \centerline{  \it Email: zhaojiefeng001@163.com}


\begin{abstract}
  In this paper we study the global regularity of the following 2D (two-dimensional) generalized
  magnetohydrodynamic equations
\begin{eqnarray*}
 \left\{
 \begin{array}{llll}
   u_t + u \cdot \nabla u & = & - \nabla p + b \cdot \nabla b - \nu (-\triangle)^{\alpha} u
  \\
  b_t + u \cdot \nabla b & = & b \cdot \nabla u - \kappa (-\triangle)^{\beta} b
 \end{array}\right.
\end{eqnarray*}
and get global regular solutions when $ 0\leqslant\alpha < 1 /
2,\,\, \beta \geqslant 1, \,\,3\alpha + 2\beta >3 $, which improves
the results in \cite{TYZ2013}. In particular, we obtain the global
regularity of the 2D generalized MHD when $\alpha=0$ and
$\beta>\frac 32$.
\end{abstract}


\section{Introduction}


Consider the  Cauchy problem of the 2D (two-dimensional) generalized
magnetohydrodynamic equations:
\begin{eqnarray}
\left\{
 \begin{array}{llll}\label{eq}
  u_t + u \cdot \nabla u  =  - \nabla p + b \cdot \nabla b - \nu (-\triangle)^{
  \alpha} u,  \\
  b_t + u \cdot \nabla b = b \cdot \nabla u - \kappa (-\triangle)^{\beta} b,\\
  \nabla \cdot u = \nabla \cdot b  =  0, \\
  u\left(x,0\right)=u_0\left(x\right),\,\,\, b\left(x,0\right)=b_0\left(x\right)
  \end{array}\right.
\end{eqnarray}
for $x\in \mathbb{R}^2$ and $t>0$, where $ u=u\left(x,t\right) $ is
the velocity field, $ b =b\left(x,t\right) $ is the magnetic field,
$ p =p\left(x,t\right) $ is the pressure, and $
u_0\left(x\right),\,b_0\left(x\right) $ with $\mathrm{div}
u_0\left(x\right)=\mathrm{div} b_0\left(x\right)=0$ are the initial
velocity and magnetic field, respectively. Here $\nu, \kappa,
\alpha, \beta \ge 0$ are nonnegative constants and
$(-\triangle)^{1/2}$ is defined through the Fourier transform by
\begin{equation*}
  ((-\triangle)^{1/2} f)^{\wedge}(\xi) = \left| \xi \right| \widehat{f}(\xi),
\end{equation*}
where $\wedge$ denotes the Fourier transform. Later, we will use the
inverse Fourier transform $\vee$. As usual, we write
$(-\triangle)^{1/2}$ as $\Lambda$. To simplify the presentation, we
will assume $\nu=\kappa=1$ when $\alpha>0$ and $\nu=0,\,\,\kappa=1$
when $\alpha=0$.

There have been extensive studies  on the global regularity of
solutions to  (\ref{eq}) (see e.g.
\cite{W2003}\cite{W2011}\cite{TYZ2013} and references therein). It
follows from \cite{W2011} that the problem (\ref{eq}) has a unique
global regular solution if
\begin{displaymath}
  \alpha \geqslant 1 , \ \  \beta > 0,
  \ \ \alpha + \beta \geqslant 2.
\end{displaymath}
Recently,  Tran, Yu and Zhai \cite{TYZ2013} proved that if
\begin{equation*}
  \alpha \geqslant 1 / 2,\,\, \beta \geqslant 1 \hspace{2em}\mbox{or}\hspace{2em}
  0 \leqslant \alpha < 1 / 2,\,\, 2 \alpha + \beta > 2 \hspace{2em}\mbox{or}\hspace{2em}
  \alpha \geqslant 2, \,\,\beta = 0,
\end{equation*}
then the solution is global regular.

In this paper, we will prove the global regularity to \eqref{eq}
when $ 0\leqslant\alpha < 1 / 2,\,\, \beta \geqslant 1, \,\,3\alpha
+ 2\beta >3 $.  Denote the vorticity by $\omega = - \partial_2 u_1 +
\partial_1 u_2$
  and the current by $j  = - \partial_2 b_1 + \partial_1 b_2$. We will prove that $\omega, j\in L^2(0,T;L^\infty)$
  and hence obtain the global regularity of the solution due to BKM type criterion (see \cite{CKS1997}). To this end,
  based on the estimates of $\omega, j$ in $L^\infty(0,T; L^2)$ in \cite{TYZ2013}, we will prove a new estimate on $\Lambda^r
  j$ for some $r>0$
  in this paper (see Lemma 2). Our result improves ones in
  \cite{TYZ2013} and in particular, if $\alpha=0,\,\,\beta>\frac{3}{2}$,
the problem (\ref{eq}) has a global regular solution.

Our main result is stated as
\begin{theorem}
   Suppose that $\left( u_0, b_0 \right)\in H^k$ with $k > \max\{2,\alpha + \beta\}$. If
 \begin{equation*}
    0\leqslant\alpha < 1 / 2,\,\, \beta \geqslant 1,\,\, 3\alpha + 2\beta >3,
  \end{equation*}
 then the Cauchy problem has a unique  global regular solution.
\end{theorem}
\begin{remark}
When $\alpha+\beta>2$, $k > 2$,  the global regularity has been
proved in \cite{TYZ2013}.
\end{remark}
\begin{remark}
For the 2D MHD equations, it remains open to prove the global
regularity when $\alpha=0,\,\,\beta=1$ or $\beta=0,\,\,\alpha=1$.
\end{remark}

\section{A Priori estimates}\label{sec:a priori estimates}
In this section, we will give  a priori estimates for $\omega$ and
$j$. For convenience, we use the same notation as in \cite{TYZ2013}.
Let $\omega$ and $j$ denote the vorticity and the current
respectively, where $\omega = \nabla^{\bot} \cdot u = - \partial_2
u_1 + \partial_1 u_2$
  and $j = \nabla^{\bot} \cdot b = - \partial_2 b_1 + \partial_1 b_2$.
Applying $\nabla^{\bot} \cdot$ to the equations (\ref{eq}), we obtain the following equations for
   $\omega$ and $j$:
  \begin{eqnarray}
    \omega_t + u \cdot \nabla \omega & = & b \cdot \nabla j - \Lambda^{2
    \alpha} \omega,  \label{eq:omega-L2}\\
    j_t + u \cdot \nabla j & = & b \cdot \nabla \omega + T \left( \nabla u,
    \nabla b \right) - \Lambda^{2 \beta} j,  \label{eq:j-L2}
  \end{eqnarray}
 where
  \begin{equation*}
    T \left( \nabla u, \nabla b \right) = {\color{black} 2 \partial_1 b_1
    \left( \partial_1 u_2 + \partial_2 u_1 \right) + 2 \partial_2 u_2  \left(
    \partial_1 b_2 + \partial_2 b_1 \right)} .
  \end{equation*}

The $L^\infty(0,T;L^2(\mathbb{R}^2))$ estimates for $\omega, j$ are
obtained in \cite{TYZ2013}, which is
\begin{lemma}(\cite{TYZ2013})
  Suppose that  $\alpha \geqslant 0\,\,\beta \geqslant 1$.
  Let $u_0, b_0 \in H^1$. For any $T > 0$, we have
  \begin{equation}
    \left\| \omega \right\|_{L^2}^2 \left( t \right) + \left\| j
    \right\|_{L^2}^2 \left( t \right) + \int_0^t \left( \left\|
    \Lambda^{\alpha} \omega \right\|_{L^2}^2 + \left\| \Lambda^{\beta} j
    \right\|_{L^2}^2 \right) \mathd \tau \leqslant C \left( T
    \right) .
  \end{equation}
\end{lemma}

The following is a key lemma of this paper which will be used later.
\begin{lemma}
  Suppose that $0\leqslant\alpha<\frac{1}{2}, \beta\geqslant1, r=\alpha+\beta-1>0 \,\,\mbox{and}\,\, k \geqslant \alpha+\beta$.
  Let $u_0, b_0 \in H^k$. Then for any $T > 0$, we have
  \begin{equation}
     \left\|\Lambda^r j\right\|_{L^2}^2 \left( t \right) + \int_0^t \left\| \Lambda^{\beta+r} j
    \right\|_{L^2}^2  \mathd \tau \leqslant C \left( u_0, b_0, T
    \right) .
  \end{equation}
\end{lemma}

\begin{proof}
Applying $\Lambda^r$ on both sides of  the equation(\ref{eq:j-L2}),
we obtain
 \begin{eqnarray}
    (\Lambda^r j)_t +\Lambda^r( u \cdot \nabla j ) =  \Lambda^r(b \cdot \nabla \omega) + \Lambda^r(T \left( \nabla u,
    \nabla b \right)) - \Lambda^{2 \beta+r} j.  \label{eq:j-Hr}
  \end{eqnarray}
Multiplying (\ref{eq:j-Hr}) by $\Lambda^r j$ and integrating with
respect to $x$ in $\mathbb{R}^2$, we obtain
\begin{eqnarray}\label{7}
    \frac{1}{2}  \frac{\mathd}{\mathd t} \left\|\Lambda^r j \right\|_{L^2}^2 +\left\|\Lambda^{\beta+r}j \right\|_{L^2}^2 &=& -\int_{\mathbbm{R}^2} \Lambda^r \left(u \cdot\nabla j \right) \Lambda^r j \mathd x
     +\int_{\mathbbm{R}^2} \Lambda^r \left(b \cdot\nabla \omega \right) \Lambda^r j \mathd x
    \nonumber\\
    & &+\int_{\mathbbm{R}^2} \Lambda^r T \left( \nabla u, \nabla b\right) \Lambda^r j \mathd x
    \nonumber\\
    &\equiv& I_1+I_2+I_3.
  \end{eqnarray}
 $I_1$ is estimated as follows:
\begin{eqnarray*}
  \left|I_1\right|
  & &=\left|\int_{\mathbbm{R}^2} \Lambda^r \left(u \cdot\nabla j \right) \Lambda^r j \mathd x\right|
  \nonumber\\
  & &=\left|\int_{\mathbbm{R}^2} \Lambda^{-1}\nabla\cdot\Lambda^\alpha \left(u j \right) \Lambda^{\beta+r}j \mathd x\right|
  \nonumber\\
  & &\leqslant \left\|\Lambda^{-1}\nabla\cdot\Lambda^\alpha(u j)\right\|_{L^2} \left\|\Lambda^{\beta+r}j\right\|_{L^2}
  \nonumber\\
  & &\leqslant C\left\|\Lambda^\alpha(u j)\right\|_{L^2} \left\|\Lambda^{\beta+r}j\right\|_{L^2}
  \nonumber\\
  & &= C\left\|\int_{\mathbbm{R}^2}|\xi|^\alpha\hat{u}(\xi-\eta)\hat{j}(\eta)\mathd \eta \right\|_{L^2}\left\|\Lambda^{\beta+r}j\right\|_{L^2}
  \nonumber\\
  & &\leqslant C \left(\left\|\int_{\mathbbm{R}^2}|\xi-\eta|^\alpha\left|\hat{u}(\xi-\eta)\right|\left|\hat{j}(\eta)\right|\mathd \eta \right\|_{L^2} + \left\|\int_{\mathbbm{R}^2}\left|\hat{u}(\xi-\eta)\right||\eta|^\alpha\left|\hat{j}(\eta)\right|\mathd \eta \right\|_{L^2}\right)\left\|\Lambda^{\beta+r}j\right\|_{L^2}
  \nonumber\\
  & &= C\left(\left\|\Lambda^\alpha \left(\left|\hat{u}\right|^{\vee}\right) \left|\hat{j}\right|^{\vee} \right\|_{L^2}+\left\|\Lambda^\alpha \left(\left|\hat{j}\right|^{\vee}\right) \left|\hat{u}\right|^{\vee} \right\|_{L^2}\right)\left\|\Lambda^{\beta+r}j\right\|_{L^2}
  \nonumber\\
  & &\leqslant C\left(\left\|\Lambda^\alpha \left(\left|\hat{u}\right|^{\vee}\right)\right\|_{L^4}\left\|\left|\hat{j}\right|^{\vee}\right\|_{L^4}+\left\|\Lambda^\alpha \left(\left|\hat{j}\right|^{\vee}\right)\right\|_{L^4}\left\|\left|\hat{u}\right|^{\vee}\right\|_{L^4}\right)\left\|\Lambda^{\beta+r} j\right\|_{L^2}
  \nonumber\\
  & &\leqslant C \left(\left\|u\right\|_{L^2}^{\frac{1}{2}-\alpha}\left\|\nabla u\right\|_{L^2}^{\frac{1}{2}+\alpha}\left\|j\right\|_{L^2}^{1-\frac{1}{2\beta}}\left\| \Lambda^\beta j \right\|_{L^2}^{\frac{1}{2\beta}}+ \left\|u\right\|_{L^2}^{\frac{1}{2}}\left\|\nabla u\right\|_{L^2}^{\frac{1}{2}} \left\|j\right\|_{L^2}^{\frac{2\beta-2\alpha-1}{2\beta}}\left\|\Lambda^\beta j\right\|_{L^2}^{\frac{1+2\alpha}{2\beta}}  \right)  \left\|\Lambda^{\beta+r} j\right\|_{L^2}
  \nonumber\\
  & &\leqslant C\epsilon \left\|\Lambda^{\beta+r} j\right\|_{L^2}^2  +  C(\epsilon)\left\|u\right\|_{L^2}^{1-2\alpha}\left\| \omega\right\|_{L^2}^{1+2\alpha}\left\|j\right\|_{L^2}^{2-\frac{1}{\beta}}\left\| \Lambda^\beta j \right\|_{L^2}^{\frac{1}{\beta}}
  \nonumber\\
  & &+ C(\epsilon)\left\|u\right\|_{L^2} \left\|\omega\right\|_{L^2} \left\|j\right\|_{L^2}^{\frac{2\beta-2\alpha-1}{\beta}}\left\|\Lambda^\beta j\right\|_{L^2}^{\frac{1+2\alpha}{\beta}},
\end{eqnarray*}
where we have used  the following Gagliardo-Nirenberg inequalities
\begin{eqnarray*}
  \left\|\Lambda^\alpha \left(\left|\hat{u}\right|^{\vee}\right)\right\|_{L^4} & \leqslant & C\left\|\left|\hat{u}\right|^{\vee}\right\|_{L^2}^{\frac{1}{2}-\alpha}\left\|\nabla \left(\left|\hat{u}\right|^{\vee}\right)\right\|_{L^2}^{\frac{1}{2}+\alpha}
  \leqslant C\left\|u\right\|_{L^2}^{\frac{1}{2}-\alpha}\left\|\nabla u\right\|_{L^2}^{\frac{1}{2}+\alpha},\\
   \left\| \left|\hat{u}\right|^{\vee}\right\|_{L^4} & \leqslant & C\left\|\left|\hat{u}\right|^{\vee}\right\|_{L^2}^{\frac{1}{2}}\left\|\nabla  \left(\left|\hat{u}\right|^{\vee}\right)\right\|_{L^2}^{\frac{1}{2}}
   \leqslant C\left\|u\right\|_{L^2}^{\frac{1}{2}}\left\|\nabla u\right\|_{L^2}^{\frac{1}{2}};\\
    \left\| \left|\hat{j}\right|^{\vee}\right\|_{L^4} & \leqslant & C\left\|\left|\hat{j}\right|^{\vee}\right\|_{L^2}^{1-\frac{1}{2\beta}}\left\| \Lambda^\beta\left( \left|\hat{j}\right|^{\vee}\right) \right\|_{L^2}^{\frac{1}{2\beta}}
    \leqslant C\left\|j\right\|_{L^2}^{1-\frac{1}{2\beta}}\left\| \Lambda^\beta j \right\|_{L^2}^{\frac{1}{2\beta}},\\
    \left\| \Lambda^\alpha \left( \left|\hat{j}\right|^{\vee}\right) \right\|_{L^4} & \leqslant & C\left\|\left|\hat{j}\right|^{\vee}\right\|_{L^2}^{\frac{2\beta-2\alpha-1}{2\beta}}\left\|\Lambda^\beta \left( \left|\hat{j}\right|^{\vee}\right) \right\|_{L^2}^{\frac{1+2\alpha}{2\beta}}
    \leqslant C\left\|j\right\|_{L^2}^{\frac{2\beta-2\alpha-1}{2\beta}}\left\|\Lambda^\beta
    j\right\|_{L^2}^{\frac{1+2\alpha}{2\beta}}.
\end{eqnarray*}
Similarly, we can deal with $I_2$ as follows.
\begin{eqnarray*}
  \left|I_2\right|
  & &=\left|\int_{\mathbbm{R}^2} \Lambda^r \left(b \cdot\nabla \omega\right) \Lambda^r j \mathd x\right|
  \nonumber\\
  & &=\left|\int_{\mathbbm{R}^2} \Lambda^{-1}\nabla\cdot\Lambda^\alpha \left(b \omega \right) \Lambda^{\beta+r}j \mathd x\right|
  \nonumber\\
  & &\leqslant \left\|\Lambda^{-1}\nabla\cdot\Lambda^\alpha(b \omega)\right\|_{L^2} \left\|\Lambda^{\beta+r}j\right\|_{L^2}
  \nonumber\\
  & &\leqslant C\left\|\Lambda^\alpha(b \omega)\right\|_{L^2} \left\|\Lambda^{\beta+r}j\right\|_{L^2}
  \nonumber\\
  & &= C\left\|\int_{\mathbbm{R}^2}|\xi|^\alpha\hat{b}(\xi-\eta)\hat{\omega}(\eta)\mathd \eta \right\|_{L^2}\left\|\Lambda^{\beta+r}j\right\|_{L^2}
  \nonumber\\
  & &\leqslant C \left(\left\|\int_{\mathbbm{R}^2}|\xi-\eta|^\alpha\left|\hat{b}(\xi-\eta)\right|\left|\hat{\omega}(\eta)\right|\mathd \eta \right\|_{L^2} + \left\|\int_{\mathbbm{R}^2}\left|\hat{b}(\xi-\eta)\right|\left|\eta\right|^\alpha\left|\hat{\omega}(\eta)\right|\mathd \eta \right\|_{L^2}\right)\left\|\Lambda^{\beta+r}j\right\|_{L^2}
  \nonumber\\
  & &= C\left(\left\|\Lambda^\alpha \left(\left|\hat{b}\right|^{\vee}\right) \left|\hat{\omega}\right|^{\vee} \right\|_{L^2}+\left\|\Lambda^\alpha \left(\left|\hat{\omega}\right|^{\vee}\right) \left|\hat{b}\right|^{\vee} \right\|_{L^2}\right)\left\|\Lambda^{\beta+r}j\right\|_{L^2}
  \nonumber\\
  & &\leqslant C\left(\left\|\Lambda^\alpha \left(\left|\hat{b}\right|^{\vee}\right) \right\|_{L^\infty} \left\|\omega \right\|_{L^2}+\left\|\left|\hat{b}\right|^{\vee} \right\|_{L^\infty}\left\|\Lambda^\alpha \omega\right\|_{L^2}\right)\left\|\Lambda^{\beta+r}j\right\|_{L^2}
  \nonumber\\
  & &  \leqslant C\left(\left\| b \right\|_{L^2}^{\frac{\beta-\alpha}{1+\beta}}\left\| \Lambda^{\beta+1}b \right\|_{L^2}^{\frac{1+\alpha}{1+\beta}} \left\|\omega \right\|_{L^2}+\left\|b \right\|_{L^2}^{\frac{r}{1+r}}\left\|\Lambda^{1+r}b \right\|_{L^2}^{\frac{1}{1+r}}\left\|\Lambda^\alpha \omega\right\|_{L^2}\right)\left\|\Lambda^{\beta+r}j\right\|_{L^2}
  \nonumber\\
  & & \leqslant C\epsilon \left\|\Lambda^{\beta+r}j\right\|_{L^2}^2 + C(\epsilon)\left( \left\| b \right\|_{L^2}^{\frac{2(\beta-\alpha)}{1+\beta}}\left\| \Lambda^\beta j \right\|_{L^2}^{\frac{2(1+\alpha)}{1+\beta}} \left\|\omega \right\|_{L^2}^2+\left\|b \right\|_{L^2}^{\frac{2r}{1+r}}\left\|\Lambda^r j \right\|_{L^2}^{\frac{2}{1+r}}\left\|\Lambda^\alpha \omega\right\|_{L^2}^2\right),
\end{eqnarray*}
where we have used the following Gagliardo-Nirenberg inequalities
\begin{eqnarray*}
    \left\|\Lambda^\alpha \left(\left|\hat{b}\right|^{\vee}\right) \right\|_{L^\infty} & \leqslant & C\left\| \left|\hat{b}\right|^{\vee} \right\|_{L^2}^{\frac{\beta-\alpha}{1+\beta}}\left\| \Lambda^{\beta+1}\left(\left|\hat{b}\right|^{\vee}\right) \right\|_{L^2}^{\frac{1+\alpha}{1+\beta}}\leqslant C\left\| b \right\|_{L^2}^{\frac{\beta-\alpha}{1+\beta}}\left\| \Lambda^{\beta+1}b \right\|_{L^2}^{\frac{1+\alpha}{1+\beta}},\\
    \left\|\left|\hat{b}\right|^{\vee} \right\|_{L^\infty}  & \leqslant & C\left\|\left|\hat{b}\right|^{\vee} \right\|_{L^2}^{\frac{r}{1+r}}\left\|\Lambda^{1+r}\left(\left|\hat{b}\right|^{\vee}\right)\right\|_{L^2}^{\frac{1}{1+r}}\leqslant
    C\left\|b \right\|_{L^2}^{\frac{r}{1+r}}\left\|\Lambda^{1+r}b \right\|_{L^2}^{\frac{1}{1+r}}.
\end{eqnarray*}
Now, we give estimate of $I_3$.
\begin{eqnarray*}
    \left|I_3\right|& &=\left|\int_{\mathbbm{R}^2} \Lambda^r T \left( \nabla u, \nabla b\right) \Lambda^r j \mathd x \right|
    \nonumber\\
    & &=\left|\int_{\mathbbm{R}^2}  T \left( \nabla u, \nabla b\right) \Lambda^{2r} j \mathd x \right|
    \nonumber\\
    & &\leqslant C \left\|\nabla u \right\|_{L^2} \left\|\nabla b \right\|_{L^\infty} \left\|\Lambda^{2r}j \right\|_{L^2}
    \nonumber\\
    & &\leqslant C \left\|\nabla u \right\|_{L^2} \left\|\nabla b \right\|_{L^2}^{\frac{\beta+r-1}{\beta+r}}\left\| \Lambda^{\beta+r+1}b \right\|_{L^2}^{\frac{1}{\beta+r}} \left\| j \right\|_{L^2}^{\frac{\beta-r}{\beta+r}}\left\| \Lambda^{\beta+r}j \right\|_{L^2}^{\frac{2r}{\beta+r}}
    \nonumber\\
    & &\leqslant C \left\|\nabla u \right\|_{L^2} \left\| j \right\|_{L^2}^{\frac{2\beta-1}{\beta+r}}\left\| \Lambda^{\beta+r}j \right\|_{L^2}^{\frac{2r+1}{\beta+r}}
    \nonumber\\
    & &\leqslant C\epsilon\left\| \Lambda^{\beta+r}j \right\|_{L^2}^2 + C(\epsilon)\left\| j \right\|_{L^2}^2 \left\| \omega\right\|_{L^2}^{\frac{2(\beta+r)}{2\beta-1}},
\end{eqnarray*}
where we have used the following Gagliardo-Nirenberg inequalities
\begin{eqnarray*}
  \left\|\nabla b \right\|_{L^\infty} & \leqslant & C\left\|\nabla b \right\|_{L^2}^{\frac{\beta+r-1}{\beta+r}}\left\| \Lambda^{\beta+r+1}b \right\|_{L^2}^{\frac{1}{\beta+r}},\\
  \left\|\Lambda^{2r} j \right\|_{L^2} & \leqslant & C\left\| j \right\|_{L^2}^{\frac{\beta-r}{\beta+r}}\left\| \Lambda^{\beta+r}j \right\|_{L^2}^{\frac{2r}{\beta+r}}.
\end{eqnarray*}
Substituting estimates of $I_1-I_3$ into \eqref{7}, we obtain
\begin{eqnarray*}
    \frac{1}{2}  \frac{\mathd}{\mathd t} \left\|\Lambda^r j \right\|_{L^2}^2 +\left\|\Lambda^{\beta+r}j \right\|_{L^2}^2
    & \leqslant & C\epsilon\left\| \Lambda^{\beta+r}j \right\|_{L^2}^2 +
     C(\epsilon)\left\|u\right\|_{L^2}^{1-2\alpha}\left\| \omega\right\|_{L^2}^{1+2\alpha}\left\|j\right\|_{L^2}^{2-\frac{1}{\beta}}\left\| \Lambda^\beta j \right\|_{L^2}^{\frac{1}{\beta}}
  \nonumber\\
  & &+ C(\epsilon)\left\|u\right\|_{L^2} \left\|\omega\right\|_{L^2} \left\|j\right\|_{L^2}^{\frac{2\beta-2\alpha-1}{\beta}}\left\|\Lambda^\beta j\right\|_{L^2}^{\frac{1+2\alpha}{\beta}}
   \nonumber\\
   & & + C(\epsilon)\left( \left\| b \right\|_{L^2}^{\frac{2(\beta-\alpha)}{1+\beta}}\left\| \Lambda^\beta j \right\|_{L^2}^{\frac{2(1+\alpha)}{1+\beta}} \left\|\omega \right\|_{L^2}^2+\left\|b \right\|_{L^2}^{\frac{2r}{1+r}}\left\|\Lambda^r j \right\|_{L^2}^{\frac{2}{1+r}}\left\|\Lambda^\alpha \omega\right\|_{L^2}^2\right)
   \nonumber\\
   & &+ C(\epsilon)\left\| j \right\|_{L^2}^2 \left\| \omega\right\|_{L^2}^{\frac{2(\beta+r)}{2\beta-1}}.
\end{eqnarray*}
 Choosing $\epsilon =\frac{1}{2C},$ we get

\begin{eqnarray}\label{8}
    \frac{\mathd}{\mathd t} \left\|\Lambda^r j \right\|_{L^2}^2 +\left\|\Lambda^{\beta+r}j \right\|_{L^2}^2
    & \leqslant &
     C(\epsilon)\left\|u\right\|_{L^2}^{1-2\alpha}\left\| \omega\right\|_{L^2}^{1+2\alpha}\left\|j\right\|_{L^2}^{2-\frac{1}{\beta}}\left\| \Lambda^\beta j \right\|_{L^2}^{\frac{1}{\beta}}
  \nonumber\\
  & &+ C(\epsilon)\left\|u\right\|_{L^2} \left\|\omega\right\|_{L^2} \left\|j\right\|_{L^2}^{\frac{2\beta-2\alpha-1}{\beta}}\left\|\Lambda^\beta j\right\|_{L^2}^{\frac{1+2\alpha}{\beta}}
   \nonumber\\
   & & + C(\epsilon)\left( \left\| b \right\|_{L^2}^{\frac{2(\beta-\alpha)}{\beta}}\left\| \Lambda^\beta j \right\|_{L^2}^{\frac{2(1+\alpha)}{1+\beta}} \left\|\omega \right\|_{L^2}^2+\left\|b \right\|_{L^2}^{\frac{2r}{1+r}}\left\|\Lambda^r j \right\|_{L^2}^{\frac{2}{1+r}}\left\|\Lambda^\alpha \omega\right\|_{L^2}^2\right)
   \nonumber\\
   & &+ C(\epsilon)\left\| j \right\|_{L^2}^2 \left\| \omega\right\|_{L^2}^{\frac{2(\beta+r)}{2\beta-1}}.
\end{eqnarray}
By assumptions of the lemma, we have
$0\leqslant\alpha<\frac{1}{2},\,\, \beta \geqslant 1,\,\,
r=\alpha+\beta-1>0 $, and hence $\frac{1}{\beta}\leqslant 1,
\frac{1+2\alpha}{\beta}\leqslant
2,\frac{2(1+\alpha)}{1+\beta}\leqslant 2,\frac{2}{1+r}\leqslant 2$.
Thus, due to Lemma 1, we have
\begin{equation*}
 \left\| \Lambda^\beta j \right\|_{L^2}^{\frac{1}{\beta}}, \left\|\Lambda^\beta j\right\|_{L^2}^{\frac{1+2\alpha}{\beta}},\left\| \Lambda^\beta j \right\|_{L^2}^{\frac{2(1+\alpha)}{1+\beta}},\left\|\Lambda^\alpha \omega\right\|_{L^2}^2\in L^1(0,T).
\end{equation*}
Using the Gronwall's inequality in \eqref{8}, we obtain
 \begin{equation*}
     \left\|\Lambda^r j\right\|_{L^2}^2 \left( t \right) + \int_0^t \left\| \Lambda^{\beta+r} j
    \right\|_{L^2}^2  \mathd \tau \leqslant C \left( u_0, b_0, T
    \right).
  \end{equation*}
The proof of the lemma is complete.

\end{proof}

\section{Proof of Theorem 1}\label{sec:proof of theorem 1}
In this section, we prove Theorem 1. We will prove that $\omega,
j\in L^1(0,T;L^\infty(\mathbb{R}^2))$ and Theorem 1 is then followed
from the BKM-type criterion. Two cases will be considered
respectively: $0<\alpha < 1 / 2, \beta \geqslant 1, 3\alpha + 2\beta
>3$ and $\alpha=0, \beta
>\frac{3}{2}$.

\vspace{3mm}

{\bf case I: $0<\alpha < 1 / 2, \beta \geqslant 1, 3\alpha + 2\beta
>3$}

We will first give $L^\infty(0,T;H^1)$ estimates for $\omega, j$.
Differentiating with respective to $x_i (i=1,2)$ on both sides of
\eqref{eq:omega-L2} and \eqref{eq:j-L2} respectively, we get
\begin{equation}
  \left( \partial_i \omega \right)_t + u \cdot \nabla \left( \partial_i \omega
  \right) = - \left( \partial_i u \right) \cdot \nabla \omega + \left(
  \partial_i b \right) \cdot \nabla j + b \cdot \nabla \left( \partial_i j
  \right) - \Lambda^{2 \alpha} \left( \partial_i \omega \right) \label{eq:omega-H1},
\end{equation}
\begin{equation}
  \left( \partial_i j \right)_t + u \cdot \nabla \left( \partial_i j \right) =
  - \left( \partial_i u \right) \cdot \nabla j + \left( \partial_i b \right)
  \cdot \nabla \omega + b \cdot \nabla \left( \partial_i \omega \right) +
  \partial_i \left( T \left( \nabla u, \nabla b \right) \right) - \Lambda^{2
  \beta} \left( \partial_i j \right) \label{eq:j-H1}.
\end{equation}
Multiplying  $\partial_i \omega$ and $\partial_i j$ on both sides of
(\ref{eq:omega-H1}) and (\ref{eq:j-H1}) respectively, integrating
with respect to $x$ in $\mathbb{R}^2$ and summing up i=1,2, we
obtain
\begin{eqnarray}
  \frac{1}{2}\frac{\mathd}{\mathd t} \left( \left\| \nabla \omega \right\|_{L^2}^2 +
  \left\| \nabla j \right\|_{L^2}^2 \right) & = & - \sum_{i=1}^{2}\int_{\mathbbm{R}^2}
  \left[ \left( \partial_i u \right) \cdot \nabla \omega \right]\partial_i \omega \mathd x +
   \sum_{i=1}^{2}\int_{\mathbbm{R}^2} \left[ \left(\partial_i b \right) \cdot \nabla j \right]
   \partial_i \omega \mathd x \nonumber\\
  &  & - \sum_{i=1}^{2}\int_{\mathbbm{R}^2} \left[ \left( \partial_i u \right) \cdot \nabla
  j \right] \partial_i j \mathd x + \sum_{i=1}^{2}\int_{\mathbbm{R}^2}\left[ \left( \partial_i b \right)
   \cdot \nabla \omega \right]\partial_i j \mathd x \nonumber\\
  &  & + \sum_{i=1}^{2}\int_{\mathbbm{R}^2} \left[ \partial_i \left( T \left( \nabla u,
  \nabla b \right) \right) \right] \partial_i j \mathd x\nonumber\\
  &  & - \left\| \Lambda^{\alpha} \nabla \omega
  \right\|_{L^2}^2 - \left\| \Lambda^{\beta} \nabla j \right\|_{L^2}^2\nonumber\\
  & \leqslant &  C \left( A_1 + A_2 + A_3
  + A_4 + A_5 \right) - \left\| \Lambda^{\alpha} \nabla \omega
  \right\|_{L^2}^2 - \left\| \Lambda^{\beta} \nabla j \right\|_{L^2}^2, \label{es:1}
\end{eqnarray}
where we have used $\nabla\cdot u=\nabla\cdot b=0$, and we denote
\begin{eqnarray*}
  A_1 & = & \int_{\mathbbm{R}^2} \left| \nabla u \right|  \left| \nabla \omega
  \right|^2 \mathd x, \\
  A_2 & = & \int_{\mathbbm{R}^2} \left| \nabla b \right|  \left| \nabla j
  \right|  \left| \nabla \omega \right| \mathd x , \\
  A_3 & = & \int_{\mathbbm{R}^2} \left| \nabla u \right|  \left| \nabla j
  \right|^2 \mathd x,  \\
  A_4 & = & \int_{\mathbbm{R}^2} \left| \nabla b \right|  \left| \nabla \omega
  \right|  \left| \nabla j \right| \mathd x,  \\
  A_5 & = & \int_{\mathbbm{R}^2} \left[ \left| \nabla^2 u \right|  \left|
  \nabla b \right| + \left| \nabla u \right|  \left| \nabla^2 b \right|
  \right]  \left| \nabla j \right| \mathd x.
\end{eqnarray*}
 $ A_2-A_5 $ can be estimated in a straight way (see also \cite{TYZ2013}), which are
\begin{eqnarray*}
    A_2=A_4 & \leqslant & C \left( \varepsilon
    \right) \left\| j \right\|_{L^2}^2 + \left\| \nabla j \right\|_{L^2}^2
    \left\| \nabla \omega \right\|_{L^2}^2 + C \varepsilon \left\| \Lambda
    \nabla j \right\|_{L^2}^2
    \nonumber\\
    & \leqslant & C \left( \varepsilon\right) \left\| j \right\|_{L^2}^2 + \left\| \nabla j \right\|_{L^2}^2
    \left\| \nabla \omega \right\|_{L^2}^2 + C \varepsilon \left\| \Lambda^\beta
    \nabla j \right\|_{L^2}^2,
    \nonumber\\
    A_3 & \leqslant & C \left(\varepsilon \right)  \left\| \omega \right\|_{L^2}^2  \left\| \nabla j
    \right\|_{L^2}^2 + C \varepsilon \left\| \Lambda \nabla j \right\|_{L^2}^2
    \nonumber\\
    & \leqslant & C \left( \varepsilon\right) \left\| j \right\|_{L^2}^2 + C \left(\varepsilon \right)  \left\| \omega \right\|_{L^2}^2  \left\| \nabla j
    \right\|_{L^2}^2 + C \varepsilon \left\| \Lambda^\beta \nabla j \right\|_{L^2}^2,
    \nonumber\\
    A_5 & \leqslant & C \left( \varepsilon\right) \left\| j \right\|_{L^2}^2+C \left(\varepsilon \right)  \left\| \omega \right\|_{L^2}^2  \left\| \nabla j\right\|_{L^2}^2  + \left\| \nabla j \right\|_{L^2}^2
    \left\| \nabla \omega \right\|_{L^2}^2 + C \varepsilon \left\| \Lambda
    \nabla j \right\|_{L^2}^2\nonumber\\
    & \leqslant & C \left( \varepsilon\right) \left\| j \right\|_{L^2}^2+C \left(\varepsilon \right)  \left\| \omega \right\|_{L^2}^2  \left\| \nabla j\right\|_{L^2}^2  + \left\| \nabla j \right\|_{L^2}^2
    \left\| \nabla \omega \right\|_{L^2}^2 + C \varepsilon \left\| \Lambda^\beta
    \nabla j \right\|_{L^2}^2.
\end{eqnarray*}

Now we deal with $A_1$.
\begin{eqnarray*}
    A_1 & = & \int_{\mathbbm{R}^2} \left| \omega \right|  \left| \nabla \omega \right|^2
    \mathd x \leqslant \left\| \omega \right\|_{L^{p}}  \left\| \nabla
    \omega \right\|_{L^{q}}^2 \nonumber\\
    & \leqslant & C \left\| \omega \right\|_{L^{p}}
    \left\| \nabla\omega \right\|_{L^2}^{2-\frac{2}{p\alpha}} \left\| \Lambda^{\alpha} \nabla \omega\right\|_{L^2}^{\frac{2}{p\alpha}}\nonumber\\
    & \leqslant & C\epsilon \left\| \Lambda^{\alpha} \nabla \omega\right\|_{L^2}^{2} + C(\epsilon)
    \left\| \omega \right\|_{L^{p}}^{\frac{p\alpha}{p\alpha-1}} \left\| \nabla\omega \right\|_{L^2}^{2},
\end{eqnarray*}
where we have used the following Gagliardo-Nirenberg inequality:
\begin{equation*}
   \left\| \nabla \omega \right\|_{L^{q}} \leqslant C \left\| \nabla\omega \right\|_{L^2}^{1-\frac{1}{p\alpha}} \left\| \Lambda^{\alpha} \nabla \omega\right\|_{L^2}^{\frac{1}{p\alpha}}.
\end{equation*}
Here
\begin{equation*}
   \frac{1}{\alpha} < p <\infty, \hspace{2em} \frac{1}{p} + \frac{2}{q} = 1,
\end{equation*}
and $p$ is to be determined later.

To  estimate $\left\| \omega \right\|_{L^{p}}$, we multiply on the
both sides of (\ref{eq:omega-L2}) by $p \left| \omega \right|^{p -
2} \omega$ and integrate with respect to $x$ in $\mathbb{R}^2$ to
obtain
\begin{equation*}
  \frac{1}{p}\frac{\mathd}{\mathd t} \left\| \omega \right\|_{L^p}^p
  + \int_{\mathbbm{R}^2} \left(\Lambda^{\alpha} \omega \right)  \left| \omega \right|^{p - 2} \omega \mathd x  \leqslant  \int_{\mathbbm{R}^2} \left| b \right|  \left| \nabla j \right|
  \left| \omega \right|^{p - 1} \mathd x,
\end{equation*}
where we have used $\nabla\cdot u = 0$.
Thanks to the following inequality (see {\cite{R1995}}{\cite{CC2004}}{\cite{W2005}}{\cite{J2005}})
\begin{equation*}
  \int_{\mathbbm{R}^2} \left( \Lambda^{\alpha} \omega \right)  \left| \omega
  \right|^{p - 2} \omega \mathd x \geqslant 0,
\end{equation*}
we get
\begin{equation}\label{12}
  \frac{\mathd}{\mathd t} \left\| \omega \right\|_{L^p} \leqslant \left\| b
  \cdot \nabla j \right\|_{L^p} \leqslant \left\| b \right\|_{L^{\infty}}
  \left\| \nabla j \right\|_{L^p}.
\end{equation}
It follows from Lemma 2 that
\begin{equation*}
  j \in L^2  \left( 0, T ; H^{\beta+r} \right).
\end{equation*}
Consequently, we have
\begin{equation*}
  b\in L^2(0,T;L^\infty)\,\,\, \mbox{and} \,\,\,\nabla j\in L^2(0,T;L^p) \left(\mbox{for some}\,\, p >\frac{1}{\alpha}\right).
\end{equation*}
In fact, if $\beta + r\geqslant 2$, i.e. $\alpha + 2\beta \geqslant
3$, we can choose $\frac{1}{\alpha}< p <\infty$, such that
\begin{equation*}
  \nabla j\in L^2(0,T;L^p);
\end{equation*}
if $\beta + r < 2$,  i.e. $\alpha + 2\beta < 3$, we have $ \frac{2}{2-(\beta+r)} = \frac{2}{3-(2\beta + \alpha)}>  \frac{1}{\alpha}$, so we can  choose  $ \frac{1}{\alpha}< p < \frac{2}{2-(\beta+r)} $, such that
\begin{equation*}
  \nabla j\in L^2(0,T;L^p),
\end{equation*}
where we have used the assumption $ 3\alpha + 2\beta >3 $ in the
theorem. Integrating on $t$ in $(0,t)$ on both sides of \eqref{12}
yields
\begin{equation}
  \left\| \omega \right\|_{L^p} \leqslant \left\| \omega_0 \right\|_{L^p} +
  \int_0^t \left\| b \right\|_{L^{\infty}}  \left\| \nabla j \right\|_{L^p}
  \mathd \tau \leqslant \left\| \omega_0 \right\|_{L^p} + \left\| b
  \right\|_{L^2 \left( 0, T ; L^{\infty} \right)}  \left\| \nabla j
  \right\|_{L^2 \left( 0, T ; L^p \right)} \leqslant C \left( \omega_0, T
  \right)\label{es:2},
\end{equation}
which implies that  $\left\| \omega \right\|_{L^p} \in L^{\infty}
\left( 0, T \right)$.

Putting the estimates of ($A_1-A_5$) into (\ref{es:1}), we have
\begin{eqnarray}
  \frac{1}{2}\frac{\mathd}{\mathd t} \left( \left\| \nabla \omega \right\|_{L^2}^2 +
  \left\| \nabla j \right\|_{L^2}^2 \right) & \leqslant &  C \left(\varepsilon \right) ( \left\| \omega \right\|_{L^2}^2 + \left\| \nabla j\right\|_{L^2}^2 +\left\| \omega \right\|_{L^{p}}^{\frac{p\alpha}{p\alpha-1}})(\left\| \nabla j \right\|_{L^2}^2+ \left\| \nabla \omega \right\|_{L^2}^2)\nonumber\\
  & & + C \left( \varepsilon\right) \left\| j \right\|_{L^2}^2
      + C\epsilon \left\| \Lambda^{\alpha} \nabla \omega\right\|_{L^2}^{2}
      +  C \varepsilon \left\| \Lambda^\beta \nabla j \right\|_{L^2}^2\nonumber\\
  & &- \left\| \Lambda^{\alpha} \nabla \omega
  \right\|_{L^2}^2 - \left\| \Lambda^{\beta} \nabla j \right\|_{L^2}^2.
\end{eqnarray}
Combining this with (\ref{es:2}) and lemma 1, taking $\epsilon$ so that $C\epsilon=\frac{1}{2}$, and utilizing the Gronwall's inequality, we get
$$
\left\| \nabla \omega \right\|_{L^2}^2 +
  \left\| \nabla j \right\|_{L^2}^2+\int_0^t \left\| \Lambda^{\alpha} \nabla \omega
  \right\|_{L^2}^2 +\left\| \Lambda^{\beta} \nabla j
  \right\|_{L^2}^2 dt\le C(T)
  $$
  which implies that
\begin{equation*}
  \nabla \omega, \nabla j \in L^{\infty} \left( 0, T ; L^2 \right), \nabla \omega \in L^2 \left( 0, T ; H^\alpha\right), \nabla j\in L^2 \left( 0, T ; H^\beta\right).
\end{equation*}
Thus, by the Sobolev imbedding, we have $\omega, j \in L^2 \left( 0,
T ; L^\infty \right)$. Applying the BKM type criterion for global
regularity (see \cite{CKS1997}), we get the proof of Theorem 1.

\vspace{3mm}

{\bf Case II: $\alpha=0, \beta >\frac{3}{2}$}

 In this case, since
$\alpha=0, \beta
>\frac{3}{2}$, we have $ r=\alpha+\beta-1 >\frac{1}{2} $ and $\beta + r >2$. It follows from Lemma 2 that
 $\nabla j \in L^2 \left( 0, T ; L^{\infty} \right)$.
Since
\begin{equation*}
    \omega_t + u \cdot \nabla \omega = b \cdot \nabla j,
\end{equation*}
we can prove that $\omega \in L^{\infty} \left( 0, T ; L^{\infty}
\right)$ by using the particle trajectory method. By taking
advantage of the BKM type criterion for global regularity (see
\cite{CKS1997}), we finish the proof of Theorem 1.

\end{document}